\documentclass[10pt]{article}
\usepackage{amssymb,amsmath}

\def\R{\mathbb{R}}

\def\stac{\stackrel}
\def\~{\stac{\sim}}
\def\8{\infty}

\def\({\left(}
\def\){\right)}

\def\be{\begin{equation}}
\def\ee{\end{equation}}

\def\ba{\begin{array}}
\def\ea{\end{array}}

\def\({\left(}\def\){\right)}

\def\to{\rightarrow}

\begin{document}
\title{
Global existence, uniform boundedness and asymptotic behavior for Gierer--Meinhardt system with critical exponent $qr=(p-1)(s+1)$
}
   \author{Xianfa Song{\thanks{E-mail:\ \tt songxianfa@tju.edu.cn (or songxianfa2004@163.com)
  }}\\
 \small Department of Mathematics, School of Mathematics, Tianjin University,\\
\small Tianjin, 300072, P. R. China\\
Zeen Song\\
\small  School of Mathematics, Statistics and Mechnics, Beijing University of Technology,\\
\small No. 100 Pingleyuan, Chaoyang District, Beijing, 100124, P. R. China
}

\maketitle
\date{}

\newtheorem{theorem}{Theorem}[section]
\newtheorem{definition}{Definition}[section]
\newtheorem{lemma}{Lemma}[section]
\newtheorem{proposition}{Proposition}[section]
\newtheorem{corollary}{Corollary}[section]
\newtheorem{remark}{Remark}[section]
\renewcommand{\theequation}{\thesection.\arabic{equation}}
\catcode`@=11 \@addtoreset{equation}{section} \catcode`@=12

\begin{abstract}
In this paper, we study the following Gierer--Meinhardt system 
\begin{equation*}
\begin{cases}
u_t=d_1\Delta u-a_1u+u^p/v^q+\delta_1(x),\quad x\in \Omega, t>0,\\
v_t=d_2\Delta v-a_2v+u^r/v^s+\delta_2(x), \quad x\in \Omega, t>0,\\
\partial_\nu u=\partial_\nu v=0,\qquad\qquad \qquad\qquad\quad x\in \partial \Omega, t>0,\\
u(x,0)=u_0(x),\quad v(x,0)=v_0(x),\quad x\in \Omega, 
\end{cases}
\end{equation*}
where $\Omega\subset \mathbb{R}^N$ is a bounded domain, $\nu$ is the outer norm vector with respect to the smooth boundary $\partial \Omega$, $d_1,d_2, p,q,r>0$, $s>-1$, $p-1<r$, $qr=(p-1)(s+1)$, 
$\delta_1(x)$, $\delta_2(x)$, $u_0(x)$ and $v_0(x)$ are continuous functions. 

For every pair $d_1,d_2>0$,  we respectively establish the global existence, uniform boundedness and asymptotic behavior for Gierer--Meinhardt system with critical exponent $qr=(p-1)(s+1)$ under certain conditions.

\end{abstract}
\noindent\textbf{Keywords.} Gierer--Meinhardt system; Global existence; Uniform boundedness; Asymptotic behavior.

\section{Introduction}\label{sec:intro}
\qquad In this paper, we consider the following Gierer--Meinhardt system
\begin{equation}\label{eq:pde}
\begin{cases}
u_t=d_1\Delta u-a_1u+u^p/v^q+\delta_1(x),\quad x\in \Omega, t>0,\\
v_t=d_2\Delta v-a_2v+u^r/v^s+\delta_2(x),\quad x\in \Omega, t>0,\\
\partial_\nu u=\partial_\nu v=0,\qquad\qquad \qquad\qquad\quad   x\in \partial \Omega, t>0,\\
u(x,0)=u_0(x),\quad v(x,0)=v_0(x),\quad  x\in \Omega. 
\end{cases}
\end{equation}
Here $\Omega\subset\R^N$ is bounded, smooth and connected domain, $\nu$ is the outer norm vector with respect to the smooth boundary $\partial \Omega$, the coefficients and exponents satisfy
\begin{align}
d_1,d_2,a_1,a_2,q,r>0,\quad p>1,\quad s>-1,\quad
p-1<r,\quad qr=(p-1)(s+1).\label{zijia1}
\end{align}
The sources $\delta_1(x)$ and $\delta_2(x)$ are nonnegative, time-independent elements of $C^\alpha(\overline\Omega)$ for some $\alpha\in(0,1)$, the initial data $u_0,v_0\in C^{2+\alpha}(\overline\Omega)$ are strictly positive and satisfy the compatible boundary conditions $\partial_\nu u_0=\partial_\nu v_0=0$.

There are many results on the global existence and blowup of the solution to Gierer--Meinhardt  system subject to Neumann boundary condition, see \cite{ CMX2020, HY2016, Jiang2006, LCQ1995, LPS2017, Song2020, ST2011, Zou1, Zou2} and the reference therein. Some authors are interested in the results on Gierer--Meinhardt system with zero Dirichlet boundary conditions, see \cite{CSX2017, Ghergu2007, Ghergu2008, Ghergu2012}. 

However, there are still many open problems on the global existence and blowup of the solution to Gierer--Meinhardt  system. In this paper, we focus on Gierer--Meinhardt  system subject to Neumann boundary condition with critical exponents $qr=(p-1)(s+1)$. Our contribution can be listed as follows. First, differing to the assumptions in \cite{LPS2017}, we needn't to require $d_1=d_2$ in this paper. For every pair $d_1,d_2>0$,  we respectively establish the global existence results on Gierer--Meinhardt system with critical exponent $qr=(p-1)(s+1)$ under certain conditions. 
Second, since global existence doesn't mean global boundedness, it is a very interesting result that the solutions are global boundedness.  Third, we obtain the decay rate of the solution when the sources in the two equations are both absent. Fourth, when the source $\delta_1\equiv 0$ while $\delta_2\not \equiv 0$, we obtain the decay rate of $u$ and the asymptotic behavior for $v$ which converges to the according stationary state $h$. 

This paper is organized as follows. In Section 2, we first establish a quotient identity,  then we can obtain the results on global existence and uniform boundedness. In Section 3, we consider the decay behavior for the solution. In Section 4, we give some examples and discussions.

\section{Global existence and uniform boundedness of the solution to (\ref{eq:pde}) }

\qquad In this section, we first establish a quotient identity,  then we can obtain the results on global existence and uniform boundedness.

First, recall the following results in \cite{LPS2017}:
\begin{equation}\label{eq:finite-positive}
u\ge (\min u_0)e^{-a_1t},\qquad v\ge\eta_0e^{-a_2t},\qquad \eta_0=\min v_0>0,
\end{equation}
for any $t\in [0,T]\subset[0,T_{\max})$, where $[0,T_{\max})$ is the maximal time interval.

Denote
\begin{align}
&\theta=\frac{p-1}{r}=\frac {q}{s+1}\in(0,1),\quad c=\frac{q}{p},
\quad K=\frac{(d_1+d_2)^2}{4d_1d_2}\ge 1,\label{eq:basic}\\
&M=p\max\{1,N/2\},\qquad
\Phi_{\theta}(y)=(1-\theta)(\theta/y)^{\theta/(1-\theta)},\qquad
H(y)=a_2y+\Phi_{\theta}(y)
\end{align}
It is easy to verify that the unique minimizer of $H(y)$ is $y_*=\theta a_2^{\theta-1}$. 

For any $m>1$, $n>M$, let
\begin{align}
&E(t)=\int_{\Omega} u^mv^{-n}dx,\quad A(x,t)=\frac{u^r}{v^{s+1}},\quad
G(x,t)=\frac{u^p}{v^q},\quad W(x,t)=u^{m/2}v^{-n/2},\\
&X(x,t)=u^{(m-2)/2}v^{-n/2}\nabla u,\qquad
Z(x,t)=u^{m/2}v^{-(n+2)/2}\nabla v.
\end{align}

Differentiating $E(t)$, then using the first equation and the second one in \ref{eq:pde}, and integrating by parts, we get the exact identity
\begin{align}\label{eq:identity}
E'(t)={}&-D_{m,n}+(na_2-ma_1)E(t)+\int_{\Omega} u^mv^{-n}(mA^\theta-nA)dx\nonumber\\
&+m\int_{\Omega}\delta_1u^{m-1}v^{-n}dx
-n\int_{\Omega}\delta_2u^mv^{-n-1}dx,
\end{align}
where
\begin{equation}\label{eq:matrix}
D_{m,n}(t)=\int_{\Omega}\sum_{j=1}^N z_j^{\mathsf T}Qz_jdx,\quad
z_j=(X_j,Z_j)^{\mathsf T},\quad
Q=\begin{pmatrix}
m(m-1)d_1&-mn(d_1+d_2)/2\\
-mn(d_1+d_2)/2&n(n+1)d_2
\end{pmatrix}.
\end{equation}
For every $\frac{n}{m}>0$ and $A\ge 0$, we can obtain
\begin{equation}\label{eq:scalar-max}
mA^\theta-nA\le m\Phi_{\theta}(\frac{n}{m})
\end{equation}
by maximizing the left-hand side at $A=(m\theta/n)^{1/(1-\theta)}$.

Now we give a result on forcing-to-$L^\infty$ transfer.

\begin{lemma}\label{lem:transfer}
Suppose that $k>\max\{1,N/2\}$ and $\sup_{t<T_{\max}}\|G(t)\|_{L^k}\le C_G$. Then both components $u$ and $v$ have upper bounds for any finite time, which means that the solution $(u,v)$ is global.
\end{lemma}

{\bf Proof:}
Let $\zeta=N/(2k)<1$ and $D_i=\|\delta_i\|_{\infty}$. By the classical theory of parabolic equation\cite[Lemma 1.3 and its proof]{Winkler2010}, we have
$$
\|e^{d_1\tau\Delta_N}f\|_{\infty}
\le C_{\Omega,k}[1+(d_1\tau)^{-\zeta}]\|f\|_k.
$$
Here $\Delta_N$ is the Laplace operator subject to Neumann boundary condition.  By Duhamel principle and $L^\infty$ contractivity for the source term, we get
\begin{equation}\label{eq:U}
\|u(t)\|_\infty\le U:=\|u_0\|_\infty+\frac{D_1}{a_1}
+C_{\Omega,k}C_G\left[\frac1{a_1}
+d_1^{-\zeta}a_1^{\zeta-1}\Gamma_{\rm Euler}(1-\zeta)\right].
\end{equation}
Set
\begin{equation}\label{eq:V}
V=\max\left\{\|v_0\|_\infty,
\left(\frac{4U^r}{a_2}\right)^{1/(s+1)},\frac{4D_2}{a_2}\right\}.
\end{equation}
It easy to verify that
$$
-a_2V+U^rV^{-s}+D_2\le-a_2V/2<0,
$$ 
which implies that the constant $V$ is a supersolution for the scalar equation with reaction $-a_2z+U^rz^{-s}+D_2$, whereas $v$ is a subsolution. Hence $v\le V$

The constants $U,V$ do not depend on $T<T_{\max}$. If $T_{\max}<\infty$,  by these upper bounds and the standard classical continuation alternative, we get a contradiction to maximality. \hfill $\Box$

{\bf Remark 2.1:} In the right hand of $\sup_{t<T_{\max}}\|G(t)\|_{L^k}\le C_G$, $C_G$ can be a time-dependent function which keeps boundeness in any finite time while it possibly tends to infinite as time tends to infinite, or it may be a constant independent of time.

Using Lemma 2.1, we can get the following result on global existence of the solution to (\ref{eq:pde}).

\begin{theorem}
Under the assumptions of (\ref{zijia1}), $\delta_1(x), \delta_2(x)\geq 0$, $u_0(x), v_0(x)>0$ for every $x\in \Omega$, and $\frac{\partial u_0}{\partial \nu}=\frac{\partial v_0}{\partial \nu}=0$ on $\partial \Omega$, there exists a unique global positive classical solution to (\ref{eq:pde}). 
\end{theorem}

{\bf Proof:} 
Since $u_0(x), v_0(x)>0$, the local classical solvability and its continuation of the solution to (\ref{eq:pde}) are the direct results of the usual local wellposedness theorem \cite{Pazy1983}.

Suppose $T_{\max}<\infty$. We first choose suitable $m, n$ in $E$.  If $K>1$, we can take any $0<n<1/(K-1)$ and then an integer
$$
m>\max\left\{M,\frac nc,\frac{n+1}{1-n(K-1)}\right\}.
$$
If $K=1$, we can take any $n>0$ and an integer $m>\max\{M,n/c,n+1\}$. In the two cases, we can guarantee 
$$
0<n<cm,\qquad mnK<(m-1)(n+1), \qquad \frac{m}{p}>\max\{1,N/2\}.
$$
Let $\eta_T=\eta_0e^{-a_2T_{\max}}>0$. Then \eqref{eq:finite-positive} implies  $v\ge\eta_T$ for every $t<T_{\max}$.   Using \eqref{eq:scalar-max} and $u^{m-1}\le1+u^m$,  we can obtain 
$$
E'\le C_1E+C_0,
\quad C_0=m\|\delta_1\|_\infty|\Omega|\eta_T^{-n},
\quad C_1=|na_2-ma_1+m\Phi_{\theta}(\frac{n}{m})|+m\|\delta_1\|_\infty
$$
from \eqref{eq:identity}. Applying Gronwall inequality, we get
 $$\sup_{t<T_{\max}}E(t)<\infty.$$
 Let $k=\frac{m}{p}$ and $g=cm-n>0$, we have
$$
\|G(t)\|_k^k=\int_{\Omega} u^mv^{-n}v^{-g}\le\eta_T^{-g}E(t).
$$
However, the inequality above implies that the solution is global, which is a contradiction. \hfill $\Box$

Furthermore, we can establish the results on uniform boundedness of the solution below. We will discuss it in two cases: 
{\bf Case 1.} At least one of $\delta_1(x)$ and $\delta_2(x)$ is nontrivial; {\bf Case 2.} $\delta_1(x)=\delta_2(x) \equiv 0$ for any $x\in \Omega$.

In order to establish uniform boundedness result, we need the following lemma.

\begin{lemma}\label{lem:floor}
If at least one of $\delta_1(x)$ and $\delta_2(x)$ is nontrivial, then there is a constant $\eta_v>0$ such that $v(x,t)\ge\eta_v$ for all $x\in \bar{\Omega}$ and $t<T_{\max}$. 
\end{lemma}

{\bf Proof:} We prove it in two subcases. Subcase 1. $\delta_2\not\equiv 0$. Let $h_2=(-d_2\Delta_N+a_2)^{-1}\delta_2$.The strong elliptic maximum principle implies $h_2>0$ on $\overline\Omega$. Let $b_2=\min\{1,\min(v_0/h_2)\}>0$. Then $b_2h_2$ is a subsolution of the linear inequality for $v$, which means $v\ge b_2h_2$.

Subcase 2. $\delta_1\not\equiv0$. Similar to the discussion in Subcase 1, we have $u\ge u_*>0$ for some constant $u^*>0$. Compare $v$ with 
$$\rho'=-a_2\rho+u_*^r\rho^{-s}, \qquad \rho(0)=\eta_0.$$ Since $s>-1$ and $s+1>0$, we obtain
$$
\rho(t)^{s+1}=\eta_0^{s+1}e^{-(s+1)a_2t}
+\frac{u_*^r}{a_2}(1-e^{-(s+1)a_2t}),
$$
which implies that  $v\ge\rho\ge\min\{\eta_0,(u_*^r/a_2)^{1/(s+1)}\}>0$.\hfill $\Box$

The following two theorems will deal with the boundedness of the solution to the system in different cases.

{\bf Case 1.} At least one of $\delta_1(x)$ and $\delta_2(x)$ is nontrivial.

\begin{theorem}\label{thm:source}
Assume that at least one of $\delta_1(x)$ and $\delta_2(x)$ is nontrivial. Choose
\begin{equation}\label{eq:closed-weights}
m>M,\qquad 0<n\le cm,\qquad mnK\le(m-1)(n+1).
\end{equation}
If $a_1>H(\frac{n}{m})$, the solution is global and uniformly bounded. Especially, if $\delta_1\equiv 0$ and $\delta_2\not\equiv0$, the sufficient condition can be relaxed as $a_1\ge H(\frac{n}{m})$. 
\end{theorem}

{\bf Proof:} Denote $\sigma=m[a_1-H(\frac{n}{m})]$ and $D_1=\|\delta_1\|_{\infty}$. By \eqref{eq:identity} and the definition of $H$, we get
$$
E'\le-\sigma E+mD_1\int_{\Omega} u^{m-1}v^{-n}dx.
$$
Young's inequality implies that
$$
u^{m-1}\le\varepsilon u^m+C_{m,\varepsilon},\qquad
C_{m,\varepsilon}=\frac{(m-1)^{m-1}}{m^m\varepsilon^{m-1}}
$$
for every $\varepsilon>0$ and every real $m>1$.

If $D_1=0$ and $\sigma\ge 0$, then $E'(t)\leq 0$ and $E(t)\le E(0)$.

If $D_1>0$ and $\sigma>0$,  we can take $\varepsilon=\sigma/(2mD_1)$, then use Lemma~\ref{lem:floor} to get
$$
E'\le-\frac\sigma2E+C_0,\qquad
\tilde{C}_0=mD_1C_{m,\varepsilon}|\Omega|\eta_v^{-n},
$$
consequently, $E(t)\le\max\{E(0),2\tilde{C}_0/\sigma\}$. 

Setting $k=\frac{m}{p}$ and $g=cm-n\ge 0$, we have
\begin{equation}\label{eq:source-forcing}
\|G\|_k^k=\int_{\Omega} u^mv^{-cm}\le\eta_v^{-g} E.
\end{equation}
Applying Lemma~\ref{lem:transfer}, we prove the conclusion. \hfill $\Box$

{\bf Case 2.} $\delta_1(x)=\delta_2(x)\equiv0$ for every $x\in \Omega$.

\begin{theorem}\label{thm:tc}
Assume that $\delta_1(x)=\delta_2(x)\equiv 0$ for every $x\in \Omega$, (\ref{zijia1}) holds and the positive initial datum $u_0$ and $v_0$ which satisfy compatible condition $\frac{\partial u_0}{\partial \nu}=\frac{\partial v_0}{\partial \nu}=0$. If $\mu=a_1-a_2c-\Phi_{\theta}(\frac{n}{m})\geq 0$, then the solution to (\ref{eq:pde}) is boundedness.  
\end{theorem}

{\bf Proof:}  We can choose $m,n$ satisfy \eqref{eq:closed-weights}, let $k=m/p$, $g=cm-n\ge 0$ and 
$$
\widehat E=e^{a_2gt}E.
$$
Using (\ref{eq:pde}), we have 
\begin{equation}\label{eq:tc-estimates}
\widehat{E}'\le-m\mu\widehat{E}.
\end{equation}
Since $\mu\ge 0$, (\ref{eq:tc-estimates}) implies that $\widehat{E}(t)\leq \widehat{E}(0)=E(0)$.  
Recalling that the lower bound \eqref{eq:finite-positive} implies that $v^{-g}\le\eta_0^{-g}e^{-a_2gt}\leq \eta_0^{-g} $ for $g\ge0$. Therefore
$$
\|G\|_k^k=\int_{\Omega} u^mv^{-n}v^{-g}dx\le\eta_0^{-g}\widehat{E}\le\eta_0^{-g}E(0).
$$
Applying Lemma~\ref{lem:transfer},  we obtain the uniform boundedness of the solution. \hfill $\Box$

In the case of $\delta_1=\delta_2\equiv0$, for some special $m,n$, we can relax the condition on $a_1$ in Theorem \ref{thm:tc} and obtain the following results.

\begin{theorem}\label{thm:companion}
Let $\delta_1=\delta_2\equiv0$. Take $m>M$ and $n=cm$. If $mnK\le(m-1)(n+1)$ and
$$
L=\frac{m(s+1)}{r}-n=\frac{mq}{p(p-1)}\ge 1, 
$$
then the solution is global and uniformly bounded if
\begin{equation}\label{eq:companion-damping}
a_1>H(\min\{c,y_*\}),\qquad y_*=\theta a_2^{\theta-1}.
\end{equation}
Moreover, there exist $\xi,\omega>0$, depending only on the chosen $m,n$ and fixed coefficients in equations, such that
\begin{equation}\label{eq:comp-decay}
F_{\xi}(t):=E(t)+\xi\int_{\Omega} v^L(x,t)dx
\le F_{\xi}(0)e^{-\omega t}
\end{equation}
and
 $$\|G(t)\|_{L^{m/p}}^{m/p}\le F_{\xi}(0)e^{-\omega t}.
 $$ 
\end{theorem}

{\bf Proof:}
Since $\inf_{0<z<c}H(z)=H(\min\{c,y_*\})$, (\ref{eq:companion-damping}) means that there is a $z_c\in(0,c)$ satisfying $H(z_c)<a_1$. Set
\begin{align}
&\widetilde{n}=mz_c,\quad\lambda=n-\widetilde{n}>0,\quad
\beta=na_2-ma_1+m\Phi_{\theta}(z_c),\\
&J=\int_{\Omega} u^{m+r}v^{-n-s-1}dx,\quad Y=\int_{\Omega} v^Ldx.
\end{align}

Applying \eqref{eq:scalar-max} to $mA^\theta-\widetilde{n}A$ in \eqref{eq:identity}, we obtain
\begin{equation}\label{eq:retained}
E'+\lambda J\le\beta E\qquad
{\rm with}\quad \beta-\lambda a_2=m[H(z_c)-a_1]<0.
\end{equation}
For $a=r/(m+r)$ and $b=m/(m+r)$, using H\"older's inequality, we get
\begin{equation}\label{eq:Holder}
E\le J^bY^a,\qquad
\int_{\Omega} u^rv^{L-s-1}dx \le J^aY^b
\end{equation}
and
\begin{equation}\label{eq:moment}
Y'=-L(L-1)d_2\int_{\Omega} v^{L-2}|\nabla v|^2dx-La_2Y
+L\int_{\Omega} u^rv^{L-s-1}dx.
\end{equation}
Meanwhile, integrating the $v$ equation, we can get
$$
\left(\int_{\Omega} vdx\right)'=-a_2\int_{\Omega} vdx+\int_{\Omega} u^rv^{-s}dx.
$$
Combining the statements above, we have
$$Y'\le-La_2Y+LJ^aY^b.$$

Let $\bar{B}=\max\{\beta,0\}<\lambda a_2$, $h=m/r>0$, and choose
\begin{equation}\label{eq:xi}
\frac{\bar{B}^{h+1}}{La_2\lambda^h}<\xi<\frac{\lambda a_2^h}{L}.
\end{equation}
Denote $X=J/Y>0$. From \eqref{eq:retained}--\eqref{eq:moment},
\begin{align}
&F'_{\xi}\le Y[\bar{B}X^b+\xi LX^a-\lambda X-\xi La_2]=-YN_{
\xi}(\tau),\label{eq:Fprime}\\
&N_\xi(\tau)=\tau^h(\lambda\tau-\bar{B})+\xi L(a_2-\tau),\qquad \tau=X^a.\label{eq:Fupper}
\end{align}

We will prove $N_\xi(\tau)>0$ for all $\tau\ge 0$ below. (i) On $\bar{B}/\lambda\le \tau\le a_2$, both terms are nonnegative and they cannot vanish together; (ii) If $\bar{B}>0$ and $0<\tau<\bar{B}/\lambda$, with $\tilde{\nu}=\bar{B}/(\lambda a_2)<1$ and $\tilde{w}=\lambda\tau/\bar{B}\in(0,1)$,
$$
\frac{\tau^h(\bar{B}-\lambda\tau)}{a_2-\tau}
=\frac{\bar{B}^{h+1}}{a_2\lambda^h}
\frac{\tilde{w}^h(1-\tilde{w})}{1-\tilde{\nu} \tilde{w}}
<\frac{\bar{B}^{h+1}}{a_2\lambda^h}<\xi L.
$$
(iii) If $\tau>a_2$, set $\tilde{w}=\tau/a_2>1$ and $\tilde{\nu}=\bar{B}/(\lambda a_2)\in[0,1)$, then
$$
\frac{\tau^h(\lambda\tau-\bar{B})}{\tau-a_2}
=\lambda a_2^h\frac{\tilde{w}^h(\tilde{w}-\tilde{\nu})}{\tilde{w}-1}
>\lambda a_2^h>\xi L.
$$
In conclusion, $N_\xi(\tau)>0$ for all $\tau\ge 0$. Moreover,
$$
\frac{N_\xi(\tau)}{\tau^h+\xi}\longrightarrow La_2>0\quad(\tau\downarrow0),
\qquad
\frac{N_\xi(\tau)}{\tau^h+\xi}\longrightarrow+\infty\quad(\tau\to\infty).
$$
By Continuity, $\frac{N_{\xi}}{\tau^h+\xi}$ has a strictly positive infimum $\omega$. Hence \eqref{eq:Fprime}--\eqref{eq:Fupper} imply $F'_{\xi}\le-\omega F_{\xi}$ and $F_{\xi}(t)\leq F_{\xi}(0)e^{-\omega t}$. Similar to the discussions in the proof of Theorem 2.3, we can obtain 
$$\|G\|_{m/p}^{m/p}\leq  F_{\xi}(0)e^{-\omega t}.$$  \hfill $\Box$

\label{sec:decay-transfer}
The results obtained above can be summed as follows.
\begin{equation}\label{eq:forcing-table}
\|G(t)\|_k\le C_Ge^{-\beta t},\qquad
\begin{array}{c|c|c}
\text{branch}&C_G&\beta\\\hline
\delta_1=\delta_2=0,\ \mu>0&(\eta_0^{-g}E(0))^{1/k}&p\mu\\
\text{zero-source companion}&F_\xi(0)^{1/k}&\omega/k\\
\delta_1=0,\ \delta_2\not\equiv0,\ \tilde{\varepsilon}>0&(\eta_v^{-g}E(0))^{1/k}&p\tilde{\varepsilon}
\end{array}
\end{equation}
Here
\begin{equation}\label{d4:eq:mu}
\mu=a_1-a_2c-\Phi_{\theta}(\frac{n}{m})>0,
\end{equation}
and in the last row
\begin{equation}\label{d4:eq:source-margin}
\tilde{\varepsilon}=a_1-H(\frac{n}{m})>0.
\end{equation}

\section{Asymptotic behavior for the solution}
\qquad In this section, we consider the asymptotic behavior for the solution.  First we will deduce the decay rate of $u$ by that of $\|G\|_k$ . 

\begin{lemma}\label{d4:lem:transfer}
Assume $\delta_1(x)\equiv0$ and $\|G(t)\|_k\le C_Ge^{-\beta t}$ with $\beta>0$ and $k>\max\{1,N/2\}$. Define
\begin{equation}\label{d4:eq:alpha-b}
\tilde{\alpha}=\min\{a_1,\beta\},\qquad \tilde{b}=\begin{cases}1,&\beta=a_1,\\0,&\beta\ne a_1.\end{cases}
\end{equation}
Then there is a constant $C_u$  such that
\begin{equation}\label{d4:eq:u-profile}
\|u(t)\|_\infty\le C_u(1+t)^{\tilde{b}} e^{-\tilde{\alpha} t}.
\end{equation}

\end{lemma}

{\bf Proof:}
Let $\zeta=N/(2k)\in(0,1)$.  By the results in \cite[Lemma 1.3 and its proof]{Winkler2010}, we have
$$
\|e^{d_1t\Delta_N}f\|_\infty
\le C_{\Omega,k}[1+(d_1t)^{-\zeta}]\|f\|_k.
$$
Using Duhamel principle, we obtain
\begin{align}\label{d4:eq:duhamel}
\|u(t)\|_\infty\le{}&\|u_0\|_\infty e^{-a_1t}\nonumber\\
&+C_{\Omega,k}C_G\int_0^t e^{-a_1(t-\tau)}
 [1+(d_1(t-\tau))^{-\zeta}]e^{-\beta\tau}d\tau.
\end{align}
If $\beta<a_1$, by elementary analysis, the integral in (\ref{d4:eq:duhamel}) can be bounded by
$$
e^{-\beta t}\left[\frac1{a_1-\beta}
+d_1^{-\zeta}(a_1-\beta)^{\zeta-1}\Gamma_{\rm E}(1-\zeta)\right],
$$
where $\Gamma_{\rm E}$ is the Euler gamma function. If $\beta=a_1$, it is
$$
e^{-a_1t}\left[t+\frac{d_1^{-\zeta}}{1-\zeta}t^{1-\zeta}\right]
\le C(1+t)e^{-a_1t}.
$$
If $\beta>a_1$, multiplying it by $e^{a_1t}$ and set $d=\beta-a_1>0$, it can be bounded by
$$
\int_0^t [1+(d_1(t-\tau))^{-\zeta}]e^{-d\tau}d\tau\leq \frac{ 1+d_1^{-\zeta}}{\min\{d_1, (1-\zeta)\}}.
$$
 In particular, for every $0<\gamma<\alpha$ there is $C_\gamma$ with $\|u(t)\|_\infty\le C_\gamma e^{-\gamma t}$.\hfill $\Box$

Now we will establish the uniform decay rates of $u$ and $v$ when sources are zero.

\label{sec:zero-asymptotics}
In convience, we define the explicit positive convolution
\begin{equation}\label{d4:eq:conv}
\mathcal{J}_{\tilde{A},\tilde{B},\rho}(t)=\int_0^t e^{-\tilde{A}(t-\tau)}(1+\tau)^\rho e^{-\tilde{B}\tau}d\tau
\end{equation}
for $\tilde{A},\tilde{B}>0$ and $\rho\ge0$. Obviously, 
\begin{equation}\label{d4:eq:conv-closed}
\mathcal{J}_{\tilde{A},\tilde{B},0}(t)=
\begin{cases}
(e^{-\tilde{B}t}-e^{-\tilde{A}t})/(\tilde{A}-\tilde{B}),&\tilde{A}\ne \tilde{B},\\
t e^{-\tilde{A}t},&\tilde{A}=\tilde{B}.
\end{cases}
\end{equation}

\begin{theorem}\label{d4:thm:zero}
Under the assumptions of (\ref{zijia1}), $\delta_1(x)=\delta_2(x)\equiv 0$, $u_0(x), v_0(x)>0$ for every $x\in \Omega$, and $\frac{\partial u_0}{\partial \nu}=\frac{\partial v_0}{\partial \nu}$.
Moreover, either \eqref{eq:closed-weights} or \eqref{d4:eq:mu}. Let $C_G,\beta>0$ be those in \eqref{eq:forcing-table} and $\tilde{\alpha},\tilde{b},C_u$ be defined in Lemma~\ref{d4:lem:transfer}, and 
\begin{equation}\label{d4:eq:AB}
\tilde{A}=(s+1)a_2,\qquad \tilde{B}=r\tilde{\alpha},\qquad \rho=\tilde{b}r.
\end{equation}
Then 
\begin{equation}\label{d4:eq:v-envelope}
\|v(t)\|_\infty^{s+1}
\le\|v_0\|_\infty^{s+1}e^{-\tilde{A}t}
 +(s+1)C_u^r\mathcal{J}_{\tilde{A},\tilde{B},\rho}(t).
\end{equation}
In particular,  there is a constant $C_v$ such that 
\begin{equation}\label{d4:eq:v-profile}
\|v(t)\|_\infty\le C_v(1+t)^\ell e^{-\tilde{a} t},\quad
\ell=\begin{cases}
 \tilde{b}r/(s+1),&\tilde{B}<\tilde{A},\\
 (\tilde{b}r+1)/(s+1),&\tilde{B}=\tilde{A},\\
 0,&\tilde{B}>\tilde{A},
\end{cases}
\end{equation}
and $\tilde{a}=\min\{a_2,r\tilde{\alpha}/(s+1)\}$.
\end{theorem}

{\bf Proof:}
Let $\mathcal U(t)=C_u(1+t)^{\tilde{b}} e^{-\tilde{\alpha} t}$ and  $z(t)$ be a spatially constant function satisfying
$$
z'=-a_2z+\mathcal U(t)^r z^{-s},\qquad z(0)=\|v_0\|_\infty.
$$
Since $u\le\mathcal U$, the function $z$ is a supersolution for the $v$ equation and $v\le z$. 

Letting $w=z^{s+1}$, we can get
$$
w'=-(s+1)a_2w+(s+1)C_u^r(1+t)^{\tilde{b}r}e^{-r\tilde{\alpha} t}
$$
and solve $w=\|v_0\|_\infty^{s+1}e^{-\tilde{A}t}
 +(s+1)C_u^r\mathcal{J}_{\tilde{A},\tilde{B},\rho}(t)$ which is exactly the right-hand side of \eqref{d4:eq:v-envelope}. We need to estimate  $\mathcal{J}_{\tilde{A},\tilde{B},\rho}(t)$ in different cases.

If $\tilde{B}<\tilde{A}$, since $(1+\tau)^\rho\leq (1+t)^\rho$ and $e^{-\tilde{A}(t-\tau)}e^{-\tilde{B}\tau}\leq e^{-\tilde{B}t}e^{(\tilde{A}-\tilde{B})\tau}$, \eqref{d4:eq:conv} implies
$$
\mathcal J_{\tilde{A},\tilde{B},\rho}(t)\le
\frac{(1+t)^\rho e^{-\tilde{B}t}}{\tilde{A}-\tilde{B}}.
$$
If $\tilde{B}=\tilde{A}$, it is exactly
$$
\frac{(1+t)^{\rho+1}-1}{\rho+1}e^{-\tilde{A}t}.
$$
If $\tilde{B}>\tilde{A}$, it is bounded by
$$
e^{-\tilde{A}t}\int_0^\infty(1+\tau)^\rho e^{-(\tilde{B}-\tilde{A})\tau}d\tau.
$$

Taking the positive $(s+1)$-th root and using (\ref{d4:eq:v-envelope}), we can obtain \eqref{d4:eq:v-profile}.
\hfill $\Box$

As a corollary of (\ref{d4:thm:zero}), we have
 
\begin{corollary}\label{d4:cor:loss}
Under the assumptions of Theorem~\ref{d4:thm:zero}, for any
$$
0<\gamma_u<\min\{a_1,\beta\},\qquad
0<\gamma_v<\min\{a_2,r\gamma_u/(s+1)\},
$$
there are constants $C_1,C_2$ such that
$\|u(t)\|_\infty\le C_1e^{-\gamma_ut}$ and
$\|v(t)\|_\infty\le C_2e^{-\gamma_vt}$.
\end{corollary}

We will give asymptotic behavior for the solution when  $\delta_1\equiv 0$ and $\delta_2\not\equiv 0$. In 
this case, let $h_2$ be as in Lemma~\ref{lem:floor}, that is,
\begin{equation}\label{d4:eq:h2}
(-d_2\Delta+a_2)h_2=\delta_2,\quad\partial_\nu h_2=0,\quad
\vartheta=\min\{1,\min(v_0/h_2)\},\quad\eta_v=\vartheta\min h_2>0
\end{equation}
and the resolvent representation is
$$
h_2(x)=\int_0^\infty e^{-a_2t}e^{d_2t\Delta_N}\delta_2\ dt.
$$
We establish the following results on convergence to the inhibitor source response.
\begin{theorem}\label{d4:thm:source}
Suppose that $\delta_1\equiv 0$, $\delta_2\not\equiv0$, and \eqref{eq:closed-weights}, \eqref{d4:eq:source-margin} hold. Then there exist constants $C_*,C_w>0$, depending on the data, fixed coefficients, source and selected weight, such that
\begin{equation}\label{d4:eq:sharp-u}
(\min u_0)e^{-a_1t}\le u(x,t)\le C_*e^{-a_1t},
\end{equation}
\begin{equation}\label{d4:eq:source-G-infty}
\|G(t)\|_\infty\le\eta_v^{-q}C_*^p e^{-pa_1t},
\end{equation}
and
\begin{equation}\label{d4:eq:w-profile}
\|v(t)-h_2\|_\infty\le
\begin{cases}
 C_we^{-\min\{a_2,ra_1\}t},&a_2\ne ra_1,\\
 C_w(1+t)e^{-a_2t},&a_2=ra_1.
\end{cases}
\end{equation}
In particular, $u\to 0$ and $v\to h_2$ uniformly. 
\end{theorem}

{\bf Proof:}
Under the assumptions of this theorem, we know $\eta_u<u\leq U$ and $\eta_v<v\le V$. By the results of  Lemma~\ref{d4:lem:transfer}, $\| u_\infty\|\le C_\gamma e^{-\gamma t}$ for any fixed $0<\gamma<\min\{a_1,p\varepsilon\}$. Consequently,
$$
G=v^{-q} u^{p-1} u
\le \eta_v^{-q}C_\gamma^{p-1}e^{-(p-1)\gamma t}u.
$$
Compare the $u$ equation with the spatially constant solution of the associated linear time-dependent equation. As $p>1$, its coefficient in excess of $-a_1$ is integrable, so
$$
\|u(t)\|_\infty
\le \|u_0\|_\infty e^{-a_1t}
\exp\left(\frac{\eta_v^{-q}C_\gamma^{p-1}}{(p-1)\gamma}
[1-e^{-(p-1)\gamma t}]\right).
$$
Letting
\begin{equation}\label{d4:eq:Cstar}
C_*=\|u_0\|_\infty
\exp\left(\frac{\eta_v^{-q}C_\gamma^{p-1}}{(p-1)\gamma}\right)
\end{equation}
and recalling \eqref{eq:finite-positive}, we know that \eqref{d4:eq:sharp-u} holds, and \eqref{d4:eq:source-G-infty} follows immediately.

Let $\mathcal{Q}=u^rv^{-s}$. If $s\ge0$,set $C_s=\eta_v^{-s}$; if $-1<s<0$, set $C_s=V^{-s}$. Then
$$
\|\mathcal{Q}(t)\|_\infty\le C_sC_*^r e^{-ra_1t}.
$$
Let $w=v-h_2$. Then $w$ satisfies
$$
w_t=d_2\Delta w-a_2w+\mathcal{Q},\qquad\partial_\nu w=0.
$$
By $L^\infty$ contractivity of the operator semigroup, we have
\begin{equation}\label{d4:eq:w-exact}
\|w(t)\|_\infty\le e^{-a_2t}\|v_0-h_2\|_\infty
+C_sC_*^r\mathcal J_{a_2,ra_1,0}(t).
\end{equation}
Substituting \eqref{d4:eq:conv-closed} into the inequality above, we prove \eqref{d4:eq:w-profile}. \hfill $\Box$

\section{Examples and discussions}
\qquad In this section, we first give some examples to illustrate our results, then we compare our results with the previous ones.

{\bf Example 1.} Asymptotic behavior for the solution when $\delta_1(x)=\delta_2(x)\equiv 0$. Let
$$
N=1,\quad p=r=2,\quad q=s=1,\quad(d_1,d_2)=(9,1),
\quad(a_1,a_2)=(4,1),\quad(m,n)=(3,1/4).
$$
Then $K=25/9$, $k=3/2$, $y=1/12$,
$$
(m-1)(n+1)-mnK=5/12,\quad g=5/4,\quad\mu=1/2.
$$
Consequently, $\beta=p\mu=1$ and by Lemma~\ref{d4:lem:transfer}, we have 
$$
\|u\|_\infty\le C_ue^{-t}.
$$
Since $\tilde{A}=(s+1)a_2=2=r\tilde{\alpha}=\tilde{B}$,
\begin{equation}\label{d4:eq:D1-rate}
\|v(t)\|_\infty\le
\left(\|v_0\|_\infty^2+2C_u^2t\right)^{1/2}e^{-t}.
\end{equation}

{\bf Example 2. } Convergence to a nonconstant inhibitor response. Assume that 
$$
\delta_1=0,\qquad\delta_2(x)=2+\cos x,\qquad u_0(x)=v_0(x)\equiv 1
$$
on $\Omega=(0,\pi)$,  the exponents, diffusion coefficients and weights are the same as those Example 1.
Then
$$
h_2(x)=2+\frac{1}{2}cos x,\quad
\vartheta=2/5,\quad\eta_v=3/5,\quad
\varepsilon=4-(1/12+3)=11/12,
$$
 $p\varepsilon=11/6$ and $m\varepsilon=11/4$. By the results of Theorem~\ref{d4:thm:source}, we get
$$
e^{-4t}\le u(x,t)\le C_*e^{-4t},\qquad
\|G(t)\|_{\infty}\le C e^{-8t},\qquad
\|v(t)-2-\frac{1}{2}cos x\|_{\infty}\le C_we^{-t}.
$$

{\bf Example 3.} A source-response resonance. Assume that $\delta_1(x)\equiv 0$, $\delta_2(x)\equiv 6$, and $u_0(x)=v_0(x)\equiv 1$.
Let $$N=1,\quad p=r=2, \quad q=s=1, \quad d_1=d_2=1, \quad a_1=3, \quad a_2=6.$$
Take $m=3,n=3/4$. Then $k=3/2$, $y=1/4$ and
$$
H(y)=5/2<3,\qquad\varepsilon=1/2,\qquad h_2=1.
$$
By the results of Theorem~\ref{d4:thm:source}, we get
$$
\|u\|_\infty\le C e^{-3t},\qquad
\|v(t)-1\|_\infty\le C(1+t)e^{-6t}.
$$

{\bf Remark 1:} To guarantee the solution decay to zero as time tends to infinite, the conditions $\mu>0$ and  $a_1>H(c)$ cannot be removed. For example, if $\delta_1=\delta_2\equiv 0$, 
$$N=1, \quad p=r=2, \quad q=s=1, \quad d_1=d_2=a_1=a_2=1,$$ and  
$$m=3, \quad n=3/2,\quad L=3/2>1,$$ but
$$
\mu=0,\qquad a_1=H(c)=H(\min\{c,y_*\})=1.
$$
Then for every $R>0$, the constant functions $u=v\equiv R$ satisfy $-R+R^2/R=0$, which means that $(u,v)=(R,R)$ is a positive constant stationary solution, it doesn't uniformly converge to zero.

{\bf Remark 2:} We give some examples to show that global existence does not imply uniform boundedness. Let
$$
 p=r=2,
 \qquad q=s=1,
 \qquad a_1=a_2=a\in(0,1),
$$
and $d_1,d_2>0$ can be arbitrarily chosen.  We can seek for the solutions which are spatially constant.

1. If $\delta_1(x)=\delta_2(x)\equiv 0$, $u_0(x)=v_0(x)\equiv z_0$, we can find $u(x,t)=v(x,t)=z(t)$ such that $(u(x,t),v(x,t))=(z(t),z(t))$ is a solution to (\ref{eq:pde}). Here 
$$
z'(t)=(1-a)z(t),
 \qquad
z(t)=z_0e^{(1-a)t}.
$$

2. If $\delta_1(x)=\delta_2(x)\equiv D_0>0$, we can seek $u(x,t)=v(x,t)=z(t)$ such that $(u(x,t),v(x,t))=(z(t),z(t))$ is a solution to (\ref{eq:pde}). Here
$$
 z'(t)=(1-a)z(t)+D_0,
$$
and
$$
 z(t)=\left(z_0+\frac{D_0}{1-a}\right)e^{(1-a)t}
 -\frac{D_0}{1-a}.
 $$

3. If $\delta_1(x)\equiv D_0>0$, $\delta_2(x)\equiv 0$, $u_0(x)>c_0=\frac{D_0}{a}$ and $v_0(x)=u_0(x)-c_0$ are constant functions, then for
$$
 u'(t)=-au(t)+\frac{u^2(t)}{u(t)-c_0}+D_0
 =(1-a)u(t)+(1+a)c_0+\frac{c_0^2}{u(t)-c_0}
$$ and 
$v(t)=u(t)-c_0$, $(u(t),v(t)$ is a solution to (\ref{eq:pde}),  $u'(t)\ge(1-a)u(t)$ and $u$ is increasing. Obviously, the solution is global, strictly positive, but unbounded.

4. If $\delta_1(x)\equiv 0$, $\delta_2(x)\equiv D_0>0$, $u_0(x)\ge c_0=\frac{2c_0}{1-a}$ and $v_0(x)=u_0(x)+c_0$, then for 
 $$
 u'(t)=-au(t)+\frac{u^2(t)}{u(t)+c_0}
 =(1-a)u(t)-c_0+\frac{c_0^2}{u(t)+c_0}.
$$ 
and $v(t)=u(t)+c_0$, $(u(t),v(t)$ is a solution to (\ref{eq:pde}), $u'(t)\geq  \frac{1-a}{2}u(t)$ and $u(t)$ is increasing. Consequently, the solution is global, strictly positive, 
but unbounded.

Now we compare our results with the previous ones. 

In \cite{AnSong2026}, An and Song  provided a general PDE comparison result requiring local Lipschitz regularity on the closed nonnegative quadrant, $f_v\ge0$, but in our paper the reaction $f$  is singular at $v=0$ and
$$
f_v=-q u^p v^{-q-1}<0\quad(u,v>0).
$$

In \cite{LPS2017}, Theorem 1.3(i), Remark 1.3(ii),  Li, Peng and Song proved global existence of the solution under the same critical exponent $qr=(p+1)(s+1)$ and $p-1<r$, but they need $d_1=d_2$ and the initial datum are strictly positive. Our results remove the equality restriction on $d_1$ and $d_2$.

In \cite{Song2020}  the second author of this paper obtained some global existence results when $p-1=r$, in \cite{CMX2020}, Chen, Melnick and Xu got some global existence results when $p-1-r>0$, $q-1-s>0$,  in \cite{Gouadria2026}, Gouadria studied  global existence topics when $(p-1)/r<\min\{q/(s+1),m/n,1\}$, but their exponent regime does not directly contain ours.

 In \cite{HY2016} Henine and Youkana established some decay results when $0<p-1<r\min\{q/(s+1),1\}$, in \cite{ST2011}, Suzuki and Takagi obtained collapse and convergence under the strict ratio $(p-1)/r<q/(s+1)$ and nonnegative $s$. Consequently neither the spatial response $h_2$ nor the full linear activator decay order in our paper is presented as a previously unknown phenomenon.

There are still many open problems on  Gierer--Meinhardt system, it needs to do further extensions of previous approach.

\end{document}